\documentclass{amsart}

\title{Computing the Mazur and Swinnerton-Dyer critical subgroup of elliptic curves}

\author{Hao Chen}

\address{Department of Mathematics\\
  University of Washington\\
  Seattle, Washington 98105}
\email{chenh123@uw.edu}

\usepackage{algorithmic}
\usepackage{algorithm}

\usepackage{float}

\DeclareMathOperator{\rank}{rank}
\DeclareMathOperator{\crit}{crit}
\DeclareMathOperator{\tr}{tr}

\usepackage{macros}
\usepackage{framed}
\usepackage{hyperref}

\begin{document}

\begin{abstract}
Let $E$ be an optimal elliptic curve defined over $\bQ$. The {\em critical subgroup} of $E$ is defined by Mazur and Swinnerton-Dyer as the subgroup of $E(\bQ)$ generated by traces of branch points under a modular parametrization of $E$. We prove that for all rank two elliptic curves with conductor smaller than 1000, the critical subgroup is torsion. First, we define a family of {\em critical polynomials} attached to $E$ and describe two algorithms to compute such polynomials. We then give a sufficient condition for the critical subgroup to be torsion in terms of the factorization of critical polynomials. Finally, a table of critical polynomials is obtained for all elliptic curves of rank two and conductor smaller than 1000, from which we deduce our result.
\end{abstract}

\maketitle

\section{Introduction}
\label{sec: intro}

\subsection{Preliminaries}
Let $E$ be an elliptic curve over $\bQ$ and let $L(E,s)$ be the  $L$-function of $E$. The rank part of the Birch and Swinnerton-Dyer (BSD) conjecture states that
\[
	\rank(E(\bQ)) = \ord_{s=1}L(E,s).
\]
The right hand side is called the {\em analytic rank of $E$}, and is denoted by $r_{\an}(E)$. The left hand side is called the {\em algebraic rank of $E$}. The rank part of the BSD conjecture is still open when $r_{\an}(E) > 1$, and its proof for $r_{\an}(E) = 1$ uses the {\em Gross-Zagier formula}, which relates the value of certain $L$-functions to heights of Heegner points.

Let $N$ be the conductor of $E$. The modular curve $X_0(N)$
is a nonsingular  projective curve defined over $\bQ$. Since $E$ is modular(Breuil, Conrad, Diamond, and Taylor \cite{breuil2001modularity}), there is a surjective morphism $\varphi:X_0(N) \to E$ defined over $\bQ$. Let $\omega_E$ be the invariant differential on $E$ and let $\omega =  \varphi^*(\omega_E)$. Then $\omega$ is a holomorphic differential on $X_0(N)$ and we have $\omega  = c f(z) dz$, where $f$ is the normalized newform attached to $E$ and
$c$ is a nonzero constant. In the rest of the paper, we fix the following notations: the elliptic curve $E$, the conductor $N$, the morphism $\varphi$, and the differential $\omega$.

Let $R_\varphi =  \sum_{[z] \in X_0(N)} (e_\varphi(z)  -1) [z]$ be the ramification divisor of $\varphi$.

\begin{Def}[Mazur and Swinnerton-Dyer \cite{mazur-sd}]
The {\em critical subgroup} of $E$ is
\[
	E_{\crit}(\bQ)  = \langle \tr(\varphi([z])) : [z] \in \supp R_\varphi \rangle,
\]
where $\tr(P) = \sum_{\sigma: \bQ(P) \to \bar{\bQ}} P^{\sigma}$.
\end{Def}

Since the divisor $R_\varphi$ is defined over $\bQ$, every point $[z]$ in its
support is in $X_0(N)(\overline{\bQ})$, hence $\varphi([z]) \in E(\overline{\bQ})$, justifying the trace operation. The group $E_{\crit}(\bQ)$ is a subgroup of $E(\bQ)$. Observe that $R_\varphi = \Div(\omega)$, thus $\deg R_{\varphi} = 2g(X_0(N))-2$. In the rest of the paper, we use the notation $\Div(\omega)$ in place of the ramification divisor $R_\varphi$. In addition, we will assume $E$ is an optimal elliptic curve, so $\varphi$ is unique up to sign. This justifies the absence of $\varphi$ in the notation $E_{\crit}(\bQ)$.

Recall the construction of {\em Heegner points}: for an imaginary quadratic order $\cO = \cO_d$ of discriminant $d < 0$,
let $H_d(x)$ denote its {\em Hilbert class polynomial}.

\begin{Def}
A point $[z] \in X_0(N)$ is a {\em ``generalized Heegner point''} if there exists a negative discriminant $d$ s.t.
$H_d(j(z)) = H_d(j(Nz)) = 0$.
If in addition we have $(d,2N) = 1$, then $[z]$ is a {\em Heegner point}.
\end{Def}

For any discriminant $d$, let $E_d$ denote the quadratic twist of $E$ by $d$. Then the Gross-Zagier formula in \cite{gross1986heegner} together with a non-vanishing theorem for $L(E_d,1)$(see, for example, Bump, Friedberg, and Hoffstein \cite{bump1990nonvanishing}) implies the following
\begin{theorem}
\label{thm0}
(1) If $r_{\an}(E) = 1$, then there exists a Heegner point $[z]$ on $X_0(N)$ such that  $\tr(\varphi([z]))$ has infinite order in $E(\bQ)$. \\
(2) If $r_{\an}(E) \geq 2$, then $\tr(\varphi([z])) \in E(\bQ)_{\tors}$ for every  ``generalized Heegner point'' $[z]$ on $X_0(N)$.
\end{theorem}
The first case in the above theorem is essential to the proof of rank BSD conjecture for $r_{\an}(E) = 1$.

Observe that the defining generators of the critical subgroup also take the form $\tr(\varphi([z]))$. Then a natural question is:

\begin{question}
Does there exist an elliptic curve $E/\bQ$ with $r_{\an}(E) \geq 2$ and $\rank(E_{\crit}(\bQ)) >0$?
\end{question}
We will show that the answer is negative
for all elliptic curves with conductor $N <1000$, using {\em critical polynomials} attached to  elliptic curves.

\subsection{Main results}
\label{sec:form of result}
Let $E, N, \varphi$, and $\omega$ be as defined previously, and write $\Div(\omega) = \sum_{[z] \in X_0(N)} n_z[z]$. Let $j$ denote the $j$-invariant function.
\begin{Def}
The {\em critical j-polynomial} of $E$ is
\[
	F_{E,j}(x) = \prod_{z \in \supp \Div(\omega), j(z) \neq \infty}(x-j(z))^{n_z}.
\]
\end{Def}
Since $\Div(\omega)$ is defined over $\bQ$ and has degree $2g(X_0(N))-2$, we have $F_{E,j}(x) \in \bQ[x]$ and $\deg F_{E,j} \leq 2g(X_0(N))-2$, where equality holds if $\Div(\omega)$ does not contain cusps. For any non-constant modular function $h \in \bQ(X_0(N))$, the {\em critical $h$-polynomial} of $E$ is defined similarly, by replacing $j$ with $h$.

In this paper we give two algorithms {\em Poly Relation} and {\em Poly Relation-YP} to
compute critical polynomials. The algorithm {\em Poly Relation} computes the critical $j$-polynomial $F_{E,j}$, and the algorithm {\em Poly Relation} computes the critical $h$-polynomial $F_{E,h}$ for some modular function $h$, chosen within the algorithm.

We then relate the critical polynomials to the critical subgroup via the following theorem. Recall that 
$H_d(x)$ denotes the Hilbert class polynomial associated to a negative discriminant $d$.
\begin{theorem}
\label{thm1}
Suppose $r_{\an}(E) \geq 2$, and assume at least one of the following holds: \\
(1) $F_{E,h}$ is irreducible for some non-constant function $h \in \bQ(X_0(N))$. \\
(2) There exists negative discriminants $D_k$ and positive integers $s_k$ for $1 \leq k \leq m$, satisfying
$\bQ(\sqrt{D_{k}}) \neq \bQ(\sqrt{D_{k'}})$ for all $k \neq k'$, and an irreducible polynomial $F_0 \in \bQ[x]$, such that 
\[
F_{E,j} = \prod_{k =1}^m H_{D_{k}}^{s_k}\cdot F_0.
\]

Then $\rank(E_{\crit}(\bQ))  = 0$.
\end{theorem}

Combining Theorem~\ref{thm1} with our computation of critical polynomials, we verified
\begin{corollary}
\label{cor2}
For all elliptic curves $E$ of rank 2 and conductor $N <1000$, the rank of $E_{\crit}(\bQ)$ is zero.
\end{corollary}

The paper is organized as follows: in Sections \ref{sec: IPR} and \ref{sec: yang pair}, we describe the algorithms {\em Poly Relation} and {\em Poly Relation-YP}. In Section \ref{sec: crit}, we prove Theorem~\ref{thm1}. Last, in Section \ref{sec: table}, we show a table of critical polynomials for all elliptic curves with rank 2 and conductor smaller than 1000, and  prove Corollary~\ref{cor2}.

\section{The algorithm {\em Poly relation}}
\label{sec: IPR}

Let $C/\bQ$ be a nonsingular projective curve. For a rational function $r \in \bQ(C)$, let $\Div_0(r)$ denote its divisor of 
zeros. We then define $\deg r=  \deg(\Div_0(r))$.
\begin{Def}
Let $C/\bQ$ be a nonsingular projective curve, and let $r, u$ be two non-constant rational functions on $C$. 
A {\em minimal polynomial relation between $r$ and $u$} is an irreducible polynomial $P(x,y) \in \bQ[x,y]$ such that $P(r,u) = 0$ and $deg_x(P) \leq \deg u, \deg_y(P) \leq \deg r$.
\end{Def}
Minimal polynomial relation always exists and is unique up to scalar multiplication.
Write $\Div(r) =  \sum n_z[z]$ and $P(x,y) = f_n(y)x^n + \cdots + f_1(y)x + f_0(y)$. We have

\begin{Prop}\label{prop: multiplicity}
If $\bQ(C) = \bQ(r,u)$ and $\gcd(f_0(y), f_n(y)) = 1$, then there is a constant $c \neq 0$ s.t.
\[
		f_0(y) = c \prod_{z \in\Div_{0}(r) \setminus \Div_{\infty}(u)} (y - u(z))^{n_z}.
\]
\end{Prop}

\begin{proof}
Dividing $P(x,y)$ by $f_n(y)$, we get $x^n + \cdots + \frac{f_0(x)}{f_n(y)}$, a minimal polynomial of $r$ over $\bQ(u)$. So $\Norm_{\bQ(r,u)/\bQ(u)}(r) = \frac{f_0(u)}{f_n(u)}$. The rest of the proof uses a theorem on extensions of valuations(see, for example,  \cite[Theorem 17.2.2]{stein2012algebraic}), which we now quote.

\begin{theorem}\label{thm: valuations}
Suppose $v$ is a nontrivial valuation on a field $K$ and let $L$ be a finite extension of $K$. Then for any $a \in L$,
\[
	\sum_{1 \leq j \leq J} w_j(a) = v(\Norm_{L/K}(a)),
\]
where the $w_j$ are normalized valuations equivalent to extensions of $v$ to $L$.
\end{theorem}

For any $z_0 \in C$ such that $u(z_0) \neq \infty$, consider the valuation $v = \ord_{(u - u(z_0))}$ on $\bQ(u)$. The set of extensions of $v$ to $\bQ(C) = \bQ(r,u)$ is in bijection with $\{ z \in C : u(z) = u(z_0) \}$. Take $a = r$ and
apply Theorem \ref{thm: valuations}, we obtain
\[
\sum_{z: u(z) = u(z_0)} \ord_z(r)  = \ord_{u-u(z_0)} \frac{f_0(u)}{f_n(u)}.
\]
Combining the identities for all $z_0 \in C \setminus \Div_\infty(u)$, we have
\[
	\prod_{z \in \Div(r): u(z) \neq \infty}{(y-u(z))^{n_z}}	=	c  \cdot \frac{f_0(y)}{f_n(y)}.
\]
If $r(z) = 0$, then the condition $\gcd(f_0(y), f_n(y)) = 1$ implies that $f_0(u(z)) = 0$ and $f_n(u(z)) \neq 0$. Therefore,
\[
		f_0(y) = c \prod_{z \in \Div_{0}(r) \setminus \Div_{\infty}(u)} (y - u(z))^{n_z}.
\]
This completes the proof.

\end{proof}

For completeness we also deal with the case where $u(z) = \infty$. The corresponding valuation is \\$\ord_{\infty}{(\frac{f}{g})} = \deg g - \deg f$, and we have
\begin{equation*}
	\sum_{z: u(z) = \infty} \ord_z(r)  = \deg f_n - \deg f_0.
\end{equation*}

We will apply Proposition \ref{prop: multiplicity} to the computation of $F_{E,j}$.
Consider $dj = j'(z)dz$, viewed as a differential on $X_0(N)$.
Fix the following two modular functions on $X_0(N)$:
\begin{equation}
\label{eq: ru}
r = j(j-1728) \frac{\omega}{dj}, \;  u = \frac{1}{j}.
\end{equation}

First we compute the divisor of $r$. Let $\cE_2(N)$ and $\cE_3(N)$ denote
the set of elliptic points of order 2 and 3 on $X_0(N)$, respectively. Then
\begin{equation}
\label{eq: divdj}
	\Div(dj) =  -j^*(\infty) - \sum_{c = cusp} c + \frac{1}{2} \left( j^*(1728) - \sum_{z \in \cE_2(N)} z \right) + \frac{2}{3} \left( j^*(0) - \sum_{z \in \cE_3(N)} z \right).
\end{equation}
Writing $j^*(\infty) = \sum_{c = cusp} e_c[c]$, we obtain
\begin{equation}
\label{eq: divr}
	\Div(r) = \Div(\omega) + \frac{1}{2} \left( j^*(1728) + \sum_{z \in \cE_2(N)} z \right) +  \frac{1}{3} \left( j^*(0) + 2\sum_{z \in \cE_3(N)} z \right)- \sum_{c = cusp} (e_c-1)[c].
\end{equation}
Note that (\ref{eq: divr}) may not be the simplified form of $\Div(r)$, due to possible cancellations when $\supp \Div(\omega)$ contains cusps. But since the definition of $F_{E,j}$ only involves critical points that are not cusps, the form of $\Div(r)$ in (\ref{eq: divr}) works fine for our purpose.

Next we show $\bQ(r,u) = \bQ(X_0(N))$ for the functions $r, u$ in (\ref{eq: ru}). First we prove a lemma.

\begin{lemma}
\label{lem: different}
Let $N > 1$ be an integer and $f \in S_2(\Gamma_0(N))$ be a newform. Suppose $\alpha \in SL_2(\bZ)$   such that  $f|[\alpha] = f$, then $\alpha \in \Gamma_0(N)$.
\end{lemma}

\begin{proof}
Write $\alpha = \abcd{a}{b}{M}{d}$. First we show that it suffices to consider the case
where $d = 1$. Since $(M,d) = 1$, there exists $y,w \in \bZ$ such that $My+dw = 1$. By replacing $(y,w)$ with $(y+kd, w-kM)$ if necessary, we may
assume $(y,N) = 1$. So we can find $x,z \in \bZ$ such that $\gamma = \abcd{x}{y}{Nz}{w} \in \Gamma_0(N)$. Now $\alpha \gamma = \abcd{*}{*}{M}{1} \in \SL_2(\bZ)$ and $f|[\alpha\gamma] = f$.

Let $w_N = \abcd{0}{-1}{N}{0}$ be the Fricke involution on $X_0(N)$. Then $f|[w_N] = \pm f$, hence $f|[w_N\alpha w_N] = f$. We compute that $w_N\alpha w_N = \abcd{-N}{M}{0}{-N}$, thus
$f(q) = f|[\abcd{-N}{M}{0}{-N}](q) = f(q \zeta_N^{-M})$, where $\zeta_N = e^{2 \pi i /N}$.
The leading term of $f(q)$ is $q$, while the leading term of  $f(q \zeta_N^{-M})$ is $\zeta_N^{-M} q$. So we must have $\zeta_N^{-M} =1$, i.e., $N \mid M$. Hence $\alpha \in \Gamma_0(N)$ and the proof is complete.
\end{proof}

\begin{Prop}
\label{prop: gen}
Let $r, u$ be as defined in (\ref{eq: ru}), then $\bQ(r,u) = \bQ(X_0(N))$.
\end{Prop}

\begin{proof}
Consider the modular curve $X(N)$ defined over the field $K =  \bQ(\mu_N)$. Its function field $K(X(N))$ is a Galois extension of $K(u)$ containing $K(X_0(N))$.
It follows that the conjugates of $r$ in the extension $K(X(N))/K(u)$ are of the form $r_i = r|[\alpha_i]$ where $ \{\alpha_i\}$ is a set of coset representatives of $\Gamma_0(N) \backslash \SL_2(\bZ)$. Note that $\bQ(r,u) = \bQ(X_0(N))$ if and only if the $r_i$ are distinct. Suppose towards contradiction that there exists $i \neq j$ such that $r|[\alpha_i]  = r|[\alpha_j]$. Since $j$ and $j'$ are invariant under the action of $\SL_2(\bZ)$, we see that $f|[\alpha_i] = f|[\alpha_j]$. Let  $\alpha = \alpha_i \alpha_j^{-1}$, then $\alpha \in \SL_2(\bZ)$ and $f[\alpha] = f$. So Lemma~\ref{lem: different} implies $\alpha \in \Gamma_0(N)$, so $\Gamma_0(N) \alpha_i = \Gamma_0(N)\alpha_j$, a contradiction.
\end{proof}

\begin{lemma}
Let $g$ be the genus of $X_0(N)$. If $T \geq 2g-2$ is a positive integer, then $rj^T$ and $u$ satisfy the second condition of Proposition \ref{prop: multiplicity}.

\end{lemma}

\begin{proof}
Let $r_1 = rj^T$. When $T \geq 2g-2$, the support of  $\Div_\infty(r_1)$ is the set of all cusps. Suppose $\gcd(f_n, f_0) >1$. Let  $p(y)$ be an irreducible factor of $\gcd(f_0,f_n)$. Consider the valuation $\ord_p$ on the field $K(y)$. Since $P$ is irreducible, there exists
an integer $i$ with $0<i<n$ such that $p \nmid f_i$.  Thus the Newton polygon of $P$ with respect to the valuation $\ord_p$ has at least one edge with negative slope and one edge with positive slope. Therefore, for any Galois extension of $L$ of $K(u)$ containing $K(r,u)$ and a valuation $\ord_{\fp}$ on $L$ extending $\ord_p$, there exists two conjugates $r', r''$ of $r$ such that $\ord_{\fp}(r') < 0$ and $\ord_{\fp}(r'') >0$. This implies that $\Div_0(r') \cap \Div_\infty(r'') \neq \emptyset$. Fix $L = K(X(N))$, then all conjugates of $r_1$  in $K(X(N))/K(u)$ are of the form $r_1(\alpha z)$ for some $\alpha \in \SL_2(\bZ)$, Hence the set of poles of any conjugate of $r_1$ is the set of all cusps on $X(N)$, a contradiction.
\end{proof}

Note that for any $T \in \bZ$, we have $\bQ(rj^T,u) = \bQ(r,u) = \bQ(X_0(N))$. Hence when $T \geq 2g-2$, the pair $(rj^T,u)$ satisfies both assumptions of Proposition \ref{prop: multiplicity}. We thus obtain
\begin{theorem}
Let $T \geq 2g-2$ be a positive integer and let $P(x,y) = f_n(y)x^n + \cdots + f_1(y)x + f_0(y)$ be a minimal polynomial relation of $rj^T$ and $u$. Then there exist integers $A$, $B$ and a nonzero constant $c$ such that
\[
		F_{E,j}(y) = c f_0(1/y) \cdot y^{A} (y - 1728)^B.
\]
The integers $A$ and $B$ are defined as follows. Let $\epsilon_i(N) = |\cE_i(N)|$ for $i$ = 2 or 3 
and let $d_N = [SL_2(\bZ): \Gamma_0(N)]$, then $A = \deg f_n - T \cdot d_N - \frac{1}{3}(d_N + 2\epsilon_3(N))$, $B = -\frac{1}{2}(d_N+\epsilon_2(N))$.
\end{theorem}

\begin{proof}
Write $\Div(\omega) =\sum n_z[z]$. Applying Proposition \ref{prop: multiplicity} to $rj^T$ and $u$, we get
\begin{equation*}
	\prod_{z: u(z) \neq 0,\infty} (y-u(z))^{n_z} \cdot (y-1/1728)^{\frac{1}{2}(d_N+\epsilon_2(N))}  = c f_0(y)
\tag{a}
\end{equation*}
and
\begin{equation*}
	\sum_{z: u(z) = \infty} \ord_z(\omega) + T \cdot d_N + \frac{1}{3}(d_N + 2\epsilon_3(N))  = \deg f_n - \deg f_0.
\tag{b}
\end{equation*}
To change from $u$ to $j$, we replace $y$ by $1/y$ in (a) and multiply both sides by $y^{\deg f_0}$ to obtain
\[
	\prod_{z: j(z) \neq 0,\infty} (y-j(z))^{n_z} \cdot (y-1728)^{\frac{1}{2}(d_N+\epsilon_2(N))}  = c f_0(1/y)y^{\deg f_0}.
\]
The contribution of $\{z \in \Div(\omega): j(z) = 0\}$ to $F_{E,j}$ can be computed from (b), so
\begin{align*}
	F_{E,j}(y)
	 &= c \cdot y^{\deg f_n - \deg f_0 -  T \cdot d_N - \frac{1}{3}(d_N + 2\epsilon_3(N))}y^{\deg f_0} \cdot (y-1728)^{-\frac{1}{2}(d_N+\epsilon_2(N))} f_0(1/y) \\
	& = c \cdot y^{\deg f_n - T \cdot d_N - \frac{1}{3}(d_N + 2\epsilon_3(N))}(y-1728)^{-\frac{1}{2}(d_N+\epsilon_2(N))} f_0(1/y).
\end{align*}
\end{proof}

Now we describe the algorithm {\em Poly Relation}.
\begin{algorithm}[H]
\caption{{\em Poly relation}}
\label{IPR}
\begin{algorithmic}[1]
    \REQUIRE $E$ = Elliptic Curve over $\bQ$; $N$ =  conductor of $E$;  $f$ = the newform attached to $E$; $g =  g(X_0(N))$, $d_N, \epsilon_2(N)$, $\epsilon_3(N)$, and $c_N$ = number of cusps of $X_0(N)$.
    \ENSURE The critical $j$-polynomial $F_{E,j}(x)$.
    \STATE  Fix a large integer $M$. $T := 2g-2$.
    \STATE  $r_1:= j^{2g-1}(j-1728)\frac{f}{j'}$, $u: = \frac{1}{j}$.
    \STATE $\deg r_1 :=  (2g-1)d_N - c_N, \deg u := d_N$.
    \STATE Compute the $q$-expansions of $r_1$ and $u$ to $q^{M}$.
    \STATE Let $\{c_{a,b}\}_{0 \leq a \leq \deg u, 0 \leq b \leq \deg r_1}$ be unknowns, compute a vector that spans the one-dimensional vector space \\
    $K$ = $\{(c_{a,b}) : \sum c_{a,b} r(q)^a u(q)^b \equiv 0 \pmod {q^M} \}$.

	\STATE $P(x,y) :=  \sum c_{a,b}x^ay^b$. Write $P(x,y) = f_n(y)x^n + \cdots + f_1(y)x + f_0(y)$.
	\STATE $A := \deg f_n - T \cdot d_N - \frac{1}{3}(d_N + 2\epsilon_3(N))$, $B :=  -\frac{1}{2}(d_N+\epsilon_2(N))$.
	\STATE Output $F_{E,j}(x) = c f_0(1/x) \cdot x^{A} (x - 1728)^B$.
\end{algorithmic}
\end{algorithm}

An upper bound on the number of terms $M$ in the above algorithm can be taken to be $2\deg r \deg u +1$, by the following lemma.
\begin{lemma}
Let $r, u \in \bQ(X_0(N))$ be non-constant functions. If there is a polynomial $P \in \bQ[x,y]$
such that $\deg_x P \leq \deg u$, $\deg_y P \leq \deg r$, and
\[
	P(r,u) \equiv 0 \pmod{q^M}
\]
for some $M > 2\deg u \deg r$,  then $P(r,u) = 0$.
\end{lemma}

\begin{proof}
Suppose $P(r,u)$ is non-constant as a rational function on $X_0(N)$, then $\deg P(r,u) \leq \deg r^{\deg u} u^{\deg r} = 2\deg u \deg r$. It follows from $P(r,u) \equiv 0 \pmod{q^M}$ that $\ord_{[\infty]}P(r,u) \geq M$. Since $M> 2\deg u \deg r$, the number of zeros of $P(r,u)$ is greater than its number of poles, a contradiction. Thus $P(r,u)$ is a constant function. But then $P(r,u)$ must be 0 since it has a zero at $[\infty]$. This completes the proof.
\end{proof}

\begin{remark}
\label{rem: in practice}
When $N$ is square free, there is a faster method that computes $F_{E,j}$ by computing
the {\em Norm} of the modular form $f$, defined as $\Norm(f) = \prod f|[A_i]$, where $\{ A_i \}$ is
a set of right coset representatives of $\Gamma_0(N)$ in $\SL_2(\bZ)$. This approach is
inspired by Ahrlgen and Ono \cite{ahlgren2003weierstrass}, where $j$-polynomials of Weierstrass points on $X_0(p)$ are computed for $p$ a prime.
\end{remark}

\begin{remark}
Also for the sake of speed, instead of taking $T = 2g-2$ in the algorithm, we may take $T = 0$. First, if $\Div(\omega)$ does not contain cusps(for example, this happens if $N$ is square free), then the functions $r$ and $u$ already satisfies the assumptions of Proposition \ref{prop: multiplicity}. Second, if $\Div(\omega)$ does contain cusps, then $\deg (r)$ will be smaller than its set value
in the algorithm, due to cancellation between zeros and poles. As a result, the vector space $K$ will have dimension greater than 1. Nonetheless, using a basis of  $K$, we could construct a set of polynomials $P_i(x,y)$ with $P_i(r,u) = 0$. Now $P(x,y)$ is the greatest common divisor of the $P_i(x,y)$.
\end{remark}

We show a table of critical $j$-polynomials. Recall that $H_d(x)$ denotes the Hilbert class polynomial associated to a negative discriminant $d$. We use Cremona's labels for elliptic curves in Table~\ref{table: small}.

\begin{table}[h]
   \centering
   \caption{Critical polynomials for some elliptic curves with conductor smaller than $100$}
   \vspace{5mm}
    \begin{tabular}{ | l | l | l |p{5cm} |}
    \hline
    $E$ & $g(X_0(N))$ & Factorization of $F_{E,j}(x)$     \\[5pt] \hline \hline
    37a & 2 & $H_{-148}(x)$ \\ \hline
    37b & 2 &  $H_{-16}(x)^2$  \\ \hline
    44a &   4  &   $H_{-44}(x)^2$ \\ \hline
    48a &  3   &  $1$ \footnotemark \\ \hline
    67a &  5  &  $x^8 + 1467499520383590415545083053760x^7 + \cdots$ \\ \hline
    89a &  7  & $H_{-356}(x)$ \\ \hline
    \end{tabular}
    \label{table: small}
   \end{table}

\footnotetext[1]{In this case $\Div(\omega) = [1/4] + [3/4] + [1/12] +[7/12]$ in supported
on cusps.}

\section{Yang pairs and the algorithm {\em Poly Relation-YP}}
\label{sec: yang pair}

The main issue with the algorithm {\em Poly Relation} is efficiency. The matrix we used to solve for $\{c_{a,b}\}$ has size roughly the conductor $N$. As $N$ gets around $10^3$, computing the matrix kernel becomes time-consuming. So a new method is needed.

We introduce an algorithm {\em Poly Relation-YP} to compute critical polynomials attached to elliptic curves. The algorithm is inspired by an idea of Yifan Yang in \cite{yang2006defining}. The algorithm {\em Poly Relation-YP} does not compute the critical $j$-polynomial. Instead, it computes a critical $h$-polynomial, where $h$ is some modular function on $X_0(N)$ chosen within the algorithm. First we restate a lemma of Yang.
\begin{lemma}[Yang \cite{yang2006defining}]
\label{lem: yanggood}
Suppose $g$, $h$ are modular functions on $X_0(N)$ with a unique pole of order $m$, $n$ at the cusp $[\infty]$, respectively, such that $\gcd(m,n) =1$. Then \\
(1) $\bQ(g,h) = \bQ(X_0(N))$. \\
(2) If the leading Fourier coefficients of $g$ and $h$ are both 1, then there is a minimal
polynomial relation between $g$ and $h$ of form
\begin{equation}
\label{eq: yang}
	y^m - x^n + \sum_{a,b \geq 0, am+bn < mn} c_{a,b}x^ay^b.
\end{equation}
\end{lemma}

Two non-constant modular functions on $X_0(N)$ are said to be a {\em Yang pair} if they satisfy the assumptions of Lemma \ref{lem: yanggood}. Following \cite{yang2006defining}, we remark that in order to find a minimal polynomial relation of a Yang pair, we can compute the Fourier expansion of $y^m - x^n$ and use products of form $x^ay^b$ to cancel the pole at $[\infty]$ until we reach zero. This approach is significantly faster than the method we used in {\em Poly Relation}, which finds a minimal polynomial relation of two arbitrary modular functions. This gain in speed is the main motivation of introducing {\em Poly Relation-YP}.

Let 
\[ 
\eta = q^{\frac{1}{24}}\prod_{n \geq 1}(1-q^n)
\] be the Dedekind $\eta$ function.
For any positive integer $d$, define the function $\eta_d$ as $\eta_d(z) = \eta(dz)$. 

An {\em $\eta$-product of level $N$} is a function of the form 
\[ 
h(z) = \prod_{d \mid N} \eta_d(z)^{r_d}
\]
where $r_d \in \bZ$ for all $d \mid N$.

The next theorem of Ligozat gives sufficient conditions for a $\eta$-product to be a modular function on $X_0(N)$.
\begin{lemma}[Ligozat's Criterion \cite{ligozat1975courbes}]
Let $h = \prod_{d \mid N} \eta_d(z)^{r_d}$ be an $\eta$-product of
level $N$. Assume the following: \\
(1)  $\sum_d r_d \frac{N}{d} \equiv 0 \pmod{24}$; (2)  $\sum_d r_d d \equiv 0 \pmod{24}$;
(3)  $\sum_d r_d = 0$; (4) $\prod_{d \mid N} (\frac{N}{d})^{r_d} \in \bQ^2$.  \\
Then $h$ is a modular function on  $X_0(N)$.
\end{lemma}

If $h \in \bQ(X_0(N))$ is an $\eta$-product, then it is a fact that the divisor $\Div(h)$ is supported on the cusps of $X_0(N)$. The next theorem allows us to construct $\eta$-products with prescribed divisors.

\begin{lemma}[Ligozat \cite{ligozat1975courbes}]
\label{lem:ligozat}
Let $N > 1$ be an integer. For every positive divisor $d \mid N$, let $(P_d)$ denote the sum of all cusps on $X_0(N)$ of denominator $d$. Let $\phi$ denote the Euler's totient function. Then there exists an explicitly computable $\eta$-product $h \in \bQ(X_0(N))$ such that
\[
\Div(h) = m_d( \, (P_d) - \phi(\gcd(d,N/d))[\infty] \,)
\]
for some positive integer $m_d$.
\end{lemma}

\begin{remark}
\label{rem: expliciteta}
By `explicitly computable' in Lemma~\ref{lem:ligozat}, we mean that one can compute a set of integers $\{r_d :  d\mid N \}$ that defines the $\eta$-product $h$ with desired property. It is a fact that the order of
vanishing of an $\eta$ product at any cusp of $X_0(N)$  is an linear combination of
the integers $r_d$. So prescribing the divisor of an $\eta$-product is equivalent to giving a
linear system on the variables $r_d$. Thus we can solve for the $r_d$'s and obtain the $q$-expansion of $h$ from the $q$-expansion of $\eta$.
\end{remark}

\begin{Prop}
\label{cor: majorize}
Let $D \geq 0 $ be a divisor on $X_0(N)$ such that $D$ is supported on the cusps. Then there exists an explicitly computable $\eta$-product $h \in \bQ(X_0(N))$ such that $\Div(h)$ is of the form $D' - m[\infty]$, where $m$ is a positive integer and $D' \geq D$.
\end{Prop}

Recall our notation from section \ref{sec: IPR} that $r = j(j-1728)\frac{\omega}{dj}$.
\begin{Prop}
\label{prop: niceproduct}
There exists an explicitly computable modular function $h \in \bQ(X_0(N))$ such that \\ (1) The functions $rh$ and $j(j-1728)h$
form a Yang pair; \; \\
 (2) $j(j-1728)h$ is zero at all cusps of $X_0(N)$ except the cusp $[\infty]$.
\end{Prop}

\begin{proof}
Let $T = \Div_\infty(j)$. Note that the support of $T$ is the set of all cusps. From (\ref{eq: divr}) we have $\Div_{\infty}(r) \leq T$, $\Div(j(j-1728)) = 2T$, $\ord_{[\infty]}(T) = 1$, and $\ord_{[\infty]}(r) = 0$. Applying Corollary \ref{cor: majorize} to the divisor
$D = 4(T-[\infty])$,  we obtain an $\eta$-product $h \in \bQ(X_0(N))$ such that $\Div(h) = D' - m[\infty]$, where $D' \geq D$. Then $\Div_{\infty}(rh) = m[\infty]$ and $\Div_{\infty}(j(j-1728)h) = (m+2)[\infty]$. If $m$ is odd, then $(m,m+2) = 1$ and (1) follows. Otherwise, we can replace $h$ by $jh$. Then a similar argument shows that $rh$ and $j(j-1728)h$ have a unique pole at $[\infty]$ and have degree $m+1$ and $m+3$, respectively. Since $m$ is even in this case, we have $(m+1, m+3) = 1$ and (1) holds.

What we just showed is the existence of an $\eta$-product $h \in \bQ(X_0(N))$ s.t. either $h$ or $jh$ satisfies (1).  Now (2) follows from the fact that $\Div_{0}(j(j-1728)h) > 2(T-[\infty])$ and $\Div_{0}(j^2(j-1728)h) > (T-[\infty])$.
\end{proof}

Let $h$ be a modular function that satisfies the conditions of Proposition~\ref{prop: niceproduct}. The next theorem allows us to compute $F_{E,j(j-1728)h}(x)$. For ease of notation, let $\tilde{r} = rh$ and $\tilde{h} = j(j-1728)h$.

\begin{theorem}
\label{thm: yangpoly}
Suppose $h$ is a modular function on $X_0(N)$ that satisfies the conditions in
Corollary  \ref{cor: niceproduct}.  Let $P(x,y)$ be a minimal polynomial relation of $\tilde{r}$ and $\tilde{h}$ of form (\ref{eq: yang}). Write  $P(x,y) = f_n(y)x^n + \cdots + f_1(y)x + f_0(y)$, and let $g$ be the genus of $X_0(N)$, then
\[
		F_{E,\tilde{h}}(x) = x^{2g-2-\deg h}f_0(x).
\]
\end{theorem}

\begin{proof}
The idea  is to apply
Proposition \ref{prop: multiplicity} to the Yang pair $(\tilde{r}, \tilde{h})$. By Lemma \ref{lem: yanggood}, every Yang pair satisfies the first assumption of Proposition \ref{prop: multiplicity}. To see the second assumption holds, observe that $f_n(y) = -1$ in (\ref{eq: yang}), so $\gcd(f_n(y),f_0(y))$ = 1. Applying Proposition \ref{prop: multiplicity}, we obtain
\[
 	f_0(y) = \prod_{z \in\Div_{0}(\tilde{r}) \setminus \Div_{\infty}(\tilde{h})} (y - \tilde{h}(z))^{n_z}.
\]
By construction of $h$, there is a divisor $D \geq 0$ on $X_0(N)$ supported on the finite set $j^{-1}(\{0,1728\}) \cup h^{-1}(0)$, such that $\Div(rh) = \Div(\omega) + D - (\deg h)[\infty]$. Taking degrees on both sides shows $\deg D = \deg h - (2g-2)$.  Since $\tilde{h}(z) = 0$ for all $z \in \supp D$, we obtain
\[
	f_0(x) = F_{E,\tilde{h}}(x) \cdot x^{\deg h -2g+2}.
\]
This completes the proof.
\end{proof}

Next we describe the algorithm {\em Poly Relation-YP}.

\begin{algorithm}[H]
\caption{{\em Poly Relation-YP}}
\begin{algorithmic} [1]
    \REQUIRE $E$ = Elliptic Curve over $\bQ$, $f$ = the newform attached to $E$.
    \ENSURE a non-constant modular function $h$ on $X_0(N)$ and the critical $\tilde{h}$-polynomial $F_{E,\tilde{h}}$, where $\tilde{h} = j(j-1728)h$.
    \STATE Find an $\eta$ product $h$ that satisfies Proposition~\ref{prop: niceproduct}.

    \STATE $\tilde{r} :=  j(j-1728)h\frac{f}{j'}$, \; $\tilde{h} := j(j-1728)h$.
    \STATE $M := (\deg \tilde{r} +1)(\deg \tilde{h} + 1)$.
    \STATE Compute $q$-expansions of $\tilde{r}$, $\tilde{h}$ to $q^{M}$.
    \STATE Compute a minimal polynomial relation $P(x,y)$ of form (\ref{eq: yang}) \\using the method mentioned after Lemma \ref{lem: yanggood}.
    \STATE Output $F_{E,\tilde{h}}(x) = x^{2g-2-\deg h}P(0,x)$.
\end{algorithmic}
\end{algorithm}

\begin{remark}
The functions $\tilde{r}$ and $\tilde{h}$ are constructed such that Theorem
 \ref{thm: yangpoly} has a nice and short statement. However, their degrees are large, which is not optimal for computational purposes. In practice, one can make different choices of two modular functions $r$ and $h$ with smaller degrees to speed up the computation. This idea is illustrated in the following example.
\end{remark}

\begin{Example}
\label{ex: 664a}
Let $E = {\bf 664a1}$ with $r_{\an}(E) = 2$. The genus $g(X_0(664)) = 81$.  Let $r_4$ be as defined in Remark~\ref{rem: in practice}. Using the method described in Remark~\ref{rem: expliciteta}, we found two $\eta$-products \\ 
\[h_1 =  (\eta_2)^{-4}(\eta_4)^6 (\eta_8)^4 (\eta_{332})^6 (\eta_{664})^{-12}, \, h_2 = (\eta_2)^{-1} (\eta_4) (\eta_{166})^{-1} (\eta_8)^2 (\eta_{332})^5 (\eta_{664})^{-6}
\]with the following properties: $h_1, h_2 \in \bQ(X_0(N))$, $\Div(rh_1) = \Div(\omega)  +  D -247[\infty]$, where $D \geq 0$ is supported on cusps, and $\Div(h_2)  = 21[1/332] + 61[1/8] + 21[1/4] - 103[\infty]$. Since (247,103) =1, the functions $rh_1$ and $h_2$ form a Yang pair. We then computed
\[
	F_{E,h_2}(x) =  x^{160} - 14434914977155584439759730967653459200865032120265600267555196444 x^{158}  + \cdots.
\]
The polynomial $F_{E,h_2}$ is irreducible in $\bQ[x]$. \end{Example}

\section{The critical subgroup $E_{crit}(\bQ)$}
\label{sec: crit}
Recall the definition of the critical subgroup for an elliptic curve $E/\bQ$: 
\[
E_{\crit}(\bQ) = \langle \tr(\varphi(e)): e \in \supp \Div(\omega)\rangle.
\]
Observe that to generate $E_{\crit}(\bQ)$, it suffices to take one representative from each Galois orbit of $\supp \Div(\omega)$. Therefore, if we let $n_{\omega}$ denote the number of Galois orbits in $\Div(\omega)$, then 
\[
\rank(E_{\crit}(\bQ)) \leq n_{\omega}.
\]
For any rational divisor $D = \sum_{[z] \in X_0(N)} n_z [z]$ on $X_0(N)$, let $p_{D} = \sum_{z \in \supp D} n_z \varphi([z])$, then $p_D \in E(\bQ)$. Note that $p_D = 0$ if $D$ is a principal divisor. The point $p_{\Div(\omega)}$ is a linear combination of the defining generators of $E_{\crit}(\bQ)$.
\begin{lemma}
\label{lem: ell}
$6 \, p_{\Div(\omega)} \equiv  - 3 \sum_{c \in \cE_2(N)} \varphi(c) - 4 \sum_{d \in \cE_3(N)}  \varphi(d) \pmod{E(\bQ)_{\tors}}$.
\end{lemma}

\begin{proof}
Let $r_0 = \omega/dj$, then $r_0 \in \bQ(X_0(N))$, hence $p_{\Div(r_0)}  =0$.
From $\Div(r_0) = \Div(\omega) - \Div(dj)$, we deduce that $p_{\Div(\omega)} = p_{\Div(dj)}$. The lemma then follows from the formula of $\Div(dj)$ given in (\ref{eq: divdj}) and the fact that the image of any cusp under $\varphi$ is torsion.
\end{proof}

\begin{Prop}
\label{prop: irr}
Assume at least one of the following holds: (1) $r_{\an}(E) \geq 2$. (2) $X_0(N)$ has no elliptic point. Then $\rank(E_{\crit}(\bQ)) \leq n_\omega - 1$.
\end{Prop}

\begin{proof}
By Lemma~\ref{lem: ell} and Theorem~\ref{thm0}, either assumption implies that $p_{\Div(\omega)}$ is torsion. But $p_{\Div(\omega)}$ is a linear combination of the $n_\omega$ generators of $E_{\crit}(\bQ)$, so these generators are linearly dependent in $E_{\crit}(\bQ) \otimes \bQ$. Hence the rank of $E_{\crit}(\bQ)$ is smaller than $n_\omega$.
\end{proof}

Now we are ready to prove Theorem~\ref{thm1}. \\
{\bf Proof of Theorem~\ref{thm1}}.
First, note that the definition of $F_{E,j}$ only involves critical points that are not cusps.
However, since images of cusps under $\varphi$ are torsion, we can replace $\Div(\omega)$ by $\Div(\omega) \setminus  \{\mbox{ cusps of } X_0(N)\}$ if necessary and assume that $\Div(\omega)$ does not contain cusps.  \\
(1) Let $d = \deg F_0$, then there exists a Galois orbit in $\Div(\omega)$
of size $d$, and the other $(2g-2-d)$ points in $\Div(\omega)$ are CM points. Let $z$ be any one of the $(2g-2-d)$  points, then  $j(z)$ is a root of $H_{D_k}(x)$ and $z \in \bQ(\sqrt{D_k})$. Since $\Div(\omega)$ is invariant under the Fricke involution $w_N$, one sees that $j(Nz)$ is also a root of $F_{E,j}$. Therefore, $j(Nz)$ is the root of $H_{D_{k'}}(x)$ for some $1 \leq k' \leq m$.  Since $z$ and $Nz$ define the same quadratic field, we must have $\bQ(\sqrt{D_k}) = \bQ(\sqrt{D_{k'}})$, which implies $k = k'$ by our assumption. It follows
that $[z]$ is a ``generalized Heegner point'' and  $\tr(\varphi([z]))$ is torsion.  By the form of $F_{E,j}$, there exists a point $[z_0] \in \supp \Div(\omega)$ such that $j(z_0)$  is a root of $F_0$. Then we have $\rank(E_{\crit}(\bQ))= \rank(\langle \tr(\varphi([z_0]) \rangle) = \rank(\langle p_{\Div(\omega)} \rangle)$. Lemma~\ref{lem: ell} implies $\langle p_{\Div(\omega)} \rangle = 0$, and it follows that $\rank(E_{\crit}(\bQ)) = 0$.

(2) If $F_{E,h}$ is irreducible,  then we necessarily have $n_\omega = 1$, and the claim follows from Proposition \ref{prop: irr}.

\begin{remark}
Christophe Delaunay has an algorithm to compute $\Div(\omega)$ numerically as equivalence classes of points in the upper half plane(see \cite{delaunay2002thesis} and \cite{delaunay2005critical}). A table of critical points for $E = {\bf 389a}$ is presented in \cite[Appendix B.1]{delaunay2002thesis}. The results suggested that $\Div(\omega)$ contains two Heegner points of discriminant 19, and the critical subgroup $E_{\crit}(\bQ)$ is torsion. Using the critical $j$-polynomial for {\bf 389a} in Table~\ref{table: rank two}, we  confirm the numerical results of Delaunay.
\end{remark}

\section{Data: critical polynomials for rank two elliptic curves}
\label{sec: table}
The columns of Table~\ref{table: rank two} are as follows. The column labeled $E$ contains Cremona labels
of elliptic curves, and those labeled $g$ contains the genus of $X_0(N)$, where $N$ is the conductor of $E$. The column labeled $h$ contains a modular function on $X_0(N)$: either the $j$ invariant or some $\eta$-product. The last column contains the factorization of the critical $h$-polynomial of $E$ defined in Section ~\ref{sec:form of result}. The factors of $F_{E,j}$ that are Hilbert class polynomials are written out explicitly. Table \ref{table: rank two} contains {\em all} elliptic curves with conductor $N \leq 1000$ and rank 2. By observing that all the critical polynomials in the table satisfy one of the assumptions of
Theorem~\ref{thm1}, we obtain Corollary~\ref{cor2}.

From our computation, it seems hard to find an elliptic curve $E/\bQ$ with $r_{\an}(E) \geq 2$ and $\rank(E_{\crit}(\bQ)) > 0$. Nonetheless, some interesting questions can be raised.

\begin{question}
For all elliptic curves $E/\bQ$, does $F_{E,j}$ always factor into a product of Hilbert class polynomials and one irreducible polynomial?
\end{question}

Yet another way to construct rational points on $E$ is to take any cusp form $g \in S_2(\Gamma_0(N), \bZ)$ and define $E_{g}(\bQ) = \langle \tr(\varphi([z]) : [z] \in \supp \Div(g(z)dz) \rangle$.
\begin{question}
Does there exist $g \in S_2(\Gamma_0(N), \bZ)$ such that $E_{g}(\bQ)$ is non-torsion?
\end{question}

\begin{remark}
Consider the irreducible factors of $F_{E,j}$ that are {\em not} Hilbert class polynomials. It turns out that their constant terms has many small
primes factors, a property also enjoyed by Hilbert class polynomials. For example, consider the polynomial  $F_{{\bf 67a}, j}$. It is irreducible and not a Hilbert class polynomial, while its constant term has factorization
\[
2^{68} \cdot 3^{2} \cdot 5^{3} \cdot 23^{6} \cdot 443^{3} \cdot 186145963^{3}.
\]
It is interesting to investigate the properties of these polynomials.
\end{remark}

\begin{remark}
The polynomial relation $P(x,y)$ between $r$ and $u$ can be applied to other computational problems regarding elliptic curves and modular forms. For example, one can use it to compute Fourier expansions of the newform $f$ at every cusp (see \cite{chen2015expansion}).
\end{remark}

   \begin{table}[!h]
   \caption{Critical polynomials for elliptic curves of rank 2 and conductor $<1000$}
   \vspace{5mm}
   \centering
    \begin{tabular}{ | l |  l  | l  |p{4.4cm}  |}
    \hline
    $E$ & $g(X_0(N))$   & $h$ & $\mbox{ Factorization of } F_{E,h}(x)$     \\ \hline \hline
    389a & 32  & $j$ & $H_{-19}(x)^2 (x^{60}+ \cdots)$ \\ \hline
    433a & 35  & $j$ &  $x^{68}+\cdots$  \\ \hline
     446d & 55  & $j$ &  $x^{108}+\cdots$ \\ \hline
    563a & 47  & $j$ &  $H_{-43}(x)^2 (x^{90} - \cdots)$   \\ \hline
    571b& 47  & $j$ &  $H_{-67}(x)^2 (x^{90} - \cdots)$ \\ \hline
    643a& 53  & $j$ &  $H_{-19}(x)^2 (x^{102} - \cdots)$ \\ \hline
    664a & 81    &   $\frac{\eta_4\eta_8^2 \eta_{332}^5}{\eta_{166}\eta_{664}^{6}{\eta_2}}$ & $x^{160} - \cdots$ \\ \hline
    655a& 65  & $j$ &  $x^{128} - \cdots$ \\  \hline
    681c& 75  & $j$ &  $x^{148} - \cdots$ \\  \hline
    707a & 67  & $j$ & $x^{132} - \cdots$  \\ \hline
    709a& 58  & $j$ &  $x^{114} - \cdots$\\ \hline
    718b& 89  & $j$ &  $ H_{-52}(x)^2 (x^{172} - \cdots)$\\ \hline
    794a& 98  & $j$ &  $H_{-4}(x)^2 (x^{192} - \cdots)$\\ \hline
    817a& 71  & $j$ &  $x^{140} - \cdots$\\ \hline
    916c & 113   & $j$ &$H_{-12}(x)^8(x^{216}+\cdots)$  \\ \hline

    944e & 115    & $\frac{\eta_{16}^4 \eta_{4}^2}{\eta_8^6}$ & $x^{224} - \cdots$ \\ \hline
    997b& 82  & $j$ &  $H_{-27}(x)^2 (x^{160} - \cdots)$\\ \hline
    997c& 82  & $j$ &  $x^{162} - \cdots$\\ \hline
    \end{tabular}
    \label{table: rank two}
   \end{table}

\bibliographystyle{plain}
\bibliography{critical-paper}

\end{document}